\documentclass[leqno, 12pt]{article}
\usepackage{amsthm,amsmath,amsfonts,amssymb,indentfirst}
\usepackage{xy} \xyoption{all}

\newtheorem{theorem}{Theorem}
\newtheorem{lemma}[theorem]{Lemma}
\newtheorem{definition}[theorem]{Definition}
\newtheorem{corollary}[theorem]{Corollary}
\newtheorem{proposition}[theorem]{Proposition}

\theoremstyle{definition}
\newtheorem{example}[theorem]{Example}
\newtheorem{remark}[theorem]{Remark}

\newcommand{\Q}{\mathbb{Q}}
\newcommand{\Z}{\mathbb{Z}}

\newcommand{\N}{\mathbb{N}}
\newcommand{\M}{\mathbb{M}}

\begin{document}

\title{Simple Lie algebras arising from  Leavitt path algebras}
\author{Gene Abrams and Zachary Mesyan}

\maketitle

\begin{abstract}  For a field $K$ and directed graph $E$, we analyze those elements of the Leavitt path algebra $L_K(E)$ which lie in the  commutator subspace $[L_K(E), L_K(E)]$.  This analysis allows us to give  easily computable necessary and sufficient conditions to determine which  Lie algebras of the form $[L_K(E), L_K(E)]$ are simple, when $E$ is row-finite (i.e., has finite out-degree)  and $L_K(E)$ is simple.

\medskip


\end{abstract}

Within the past few years, the Leavitt path algebra $L_K(E)$ of a  graph $E$ with coefficients in the field $K$ has received much attention throughout both the algebra and analysis communities.  As it turns out, quite often the algebraic properties of $L_K(E)$ (for example: simplicity, chain conditions, primeness, primitivity, stable rank) depend solely on the structure of the graph $E$, and not at all on the structure of the field $K$ (to wit, neither  on the cardinality of $K$, nor on the characteristic of $K$).

With each associative $K$-algebra $R$ one may construct the {\it Lie  K-algebra} (or {\it commutator K-algebra}) $[R,R]$ {\it of } $R$, consisting of all $K$-linear combinations of elements of the form $xy-yx$ where $x,y\in R$.   Then $[R,R]$ becomes a (non-associative) Lie algebra under the  operation $[x,y] = xy - yx$ for $x,y\in R$.  In particular, when $R=L_K(E)$, one may construct and subsequently investigate  the Lie algebra $[L_K(E), L_K(E)]$.  Such an analysis was carried out in \cite{AF} in the case where $E$ is a graph having one vertex and $n\geq 2$ loops.  In  \cite[Theorem 3.4]{AF} necessary and sufficient conditions on $n$ and  the characteristic of $K$ are given which determine  the simplicity of the Lie algebra $[L_K(E), L_K(E)]$ in this situation.    In light of the comments made above, it is of interest to note that the characteristic of $K$ does indeed play a role in this result.

There are two main contributions made in the current article.    First, we analyze various elements of $L_K(E)$ which lie in the  subspace $[L_K(E), L_K(E)]$, and in particular give in Theorem~\ref{LPA-comm} necessary and sufficient conditions for when an arbitrary linear combination of vertices of $E$ (for instance, $1_{L_K(E)}$) is such.   Second, we extend  \cite[Theorem 3.4]{AF} to all simple Leavitt path algebras arising from row-finite graphs (i.e., graphs having finite out-degree) by giving, in Corollary~\ref{LPA-simplenonunital} and Theorem~\ref{LPA-simple},  necessary and sufficient conditions on $E$ and $K$ which determine the simplicity of the Lie  $K$-algebra $[L_K(E), L_K(E)]$.

In addition, we achieve a number of supporting results which are of independent interest.  In Proposition~\ref{simple} we give necessary and sufficient conditions which determine  when a matrix ring over a simple unital algebra has a simple associated Lie algebra.  In Example~\ref{ExtendedExamples} we present, for each prime $p$, an infinite class of nonisomorphic simple Leavitt path algebras whose associated Lie algebras are simple. Moreover, these Leavitt path algebras are not isomorphic to the examples presented in \cite{AF}, showing that the current investigation does indeed extend previously known results.   In Theorem~\ref{secondrewordedCorollary} we recast Theorem~\ref{LPA-simple} in the context of $K$-theory. As a result, we observe in Proposition~\ref{Kzeropairdetermines} that for two purely infinite simple Leavitt path algebras whose Grothendieck groups correspond appropriately, the Lie algebras associated to these two algebras are either both simple or both non-simple.



\section{Lie  rings of associative rings}\label{IntroSection}

Throughout, the letters $R$ and $S$ will denote  associative (but not necessarily unital) rings, and $K$ will denote a field.  The center of the ring $R$ will be denoted by $Z(R)$.  Given a ring $R$ and two elements $x,y \in R$, we let $[x, y]$ denote the commutator $xy - yx$, and let $[R, R]$ denote the additive subgroup of $R$ generated by the commutators.   Then $[R,R]$ is a (non-associative) Lie ring, with operation $x \ast y = [x,y] = xy-yx$, which we call the {\it  Lie ring associated to} $R$.    If $R$ is in addition an algebra over a field $K$, then $[R,R]$ is a $K$-subspace of $R$ (since $k[x,y] = [kx,y]$), and in this way becomes a (non-associative) Lie $K$-algebra, which we call the {\it Lie K-algebra associated to} $R$.  Clearly $[R,R] = \{0\}$ if and only if $R$ is commutative. 

For a $d\times d$ matrix $A \in \mathbb{M}_d (R)$,  trace$(A)$ denotes as usual the sum of the diagonal entries of $A$.    
We will utilize the following  fact about traces.

\begin{proposition}[Corollary 17 from~\cite{ZM}]\label{trace}
Let $R$ be a unital ring, $d$ a positive integer, and $A \in
\mathbb{M}_d(R)$.  Then $A \in [\mathbb{M}_d(R),
\mathbb{M}_d(R)]$ if and only if $\, \mathrm{trace}(A) \in [R, R]$.   $($In particular, any $A \in \mathbb{M}_d(R)$ of trace zero is necessarily in  $\,[\mathbb{M}_d(R), \mathbb{M}_d(R)]$.$)$
\end{proposition}

Let $L$ denote a Lie ring (respectively, Lie $K$-algebra). A subset $I$ of $L$ is called a \emph{Lie ideal} if $I$ is an additive subgroup (respectively, $K$-subspace) of $L$ and $[L, I] \subseteq I$. The Lie ring (respectively, Lie $K$-algebra) $L$ is called \emph{simple} if $[L, L] \neq 0$ and the only Lie ideals of $L$ are $0$ and $L$.

While the following fact is well known, for completeness we include a proof, since we were unable to find one in the literature.

\begin{lemma}[see page 34 of \cite{JJ}] \label{simpleasLieringiffsimpleasLiealgebra} 
Let $K$ be a field and $L$ a Lie $K$-algebra.  Then $L$ is simple as a Lie ring if and only if L is simple as  Lie $K$-algebra.
\end{lemma}

\begin{proof}  We only show that simplicity as a Lie $K$-algebra implies simplicity as a Lie ring, since the other direction is trivial.  So suppose that $I$ is a nonzero ideal of $L$ in the Lie ring sense (i.e., we do not assume that $I$ is a $K$-subspace of $L$). We seek to show that $I = L$.  Since $[L,I] \subseteq I$, it is easy to see that the additive subgroup $[L,I]$ of $L$ is a Lie ideal (in the Lie ring sense) of $L$. Since $L$ is simple, the center of $L$ is zero, which yields that $[L,I] \neq \{0\}$.  But for any $k\in K$, $i\in I$,  and $\ell \in L$ we have $k[\ell,i] = [k\ell,i] \in [L,I]$, showing that $[L,I]$ is a $K$-subspace of $L$.  By the simplicity of $L$ as a Lie algebra, this gives $[L,I] = L$, and since $[L,I] \subseteq I$, we have $I = L$, as desired.
\end{proof}

As a consequence of Lemma~\ref{simpleasLieringiffsimpleasLiealgebra}, throughout the article we will often use the concise phrase ``$L$ is simple" to indicate that the Lie $K$-algebra $L$ is simple either as a Lie ring or as a Lie $K$-algebra.    The following result of Herstein will play a pivotal role in our analysis.

\begin{theorem}[Theorem 1.13 from \cite{Herstein}] \label{Herst}
Let $S$ be a simple ring. Assume either that $\, \mathrm{char}(S) \neq 2$, or that $S$ is not $\, 4$-dimensional over $Z(S)$, where $Z(S)$ is a field. Then $U \subseteq Z(S)$ for any proper Lie ideal $U$ of the Lie ring $\, [S,S]$.
\end{theorem}

\begin{corollary} \label{Herst-Cor}
Let $R$ be a simple ring, $d$ a positive integer, and $S=\M_d(R)$. If $Z(R)=0$, then either the Lie ring $\, [S,S]$ is simple, or $\, [ [S,S], [S,S]]=0$.
\end{corollary}

\begin{proof}
If $R$ is a simple ring, then so is $S$. The result now follows from Theorem~\ref{Herst} upon noting that if $0=Z(R)=Z(S)$ (where we identify $R$ with its diagonal embedding in $S$), then $S$ cannot be $4$-dimensional over $Z(S)$.
\end{proof}

\begin{lemma}\label{intersection}
Let $R$ be a ring, $d \geq 2$ an integer, and $S=\M_d(R)$. If $Z(R) \cap [R,R] \neq 0$, then the Lie ring $\, [S,S]$ is not simple.
\end{lemma}

\begin{proof}
Let $a \in Z(R) \cap [R,R]$ be any nonzero element, and let $A \in S$ be the matrix $\mathrm{diag}(a)$ (having $a$ as each entry on the main diagonal and zeros elsewhere). Write $a = \sum_{i=1}^n [b_i,c_i]$ for some $b_i,c_i \in R$, and set $B_i = \mathrm{diag}(b_i)$ and $C_i = \mathrm{diag}(c_i)$. Then $A = \sum_{i=1}^n[B_i,C_i]$ is a nonzero element of $[S,S]$. Since $A \in Z(S)$, the additive subgroup generated by $A$ is a nonzero Lie ideal of $[S,S]$. This Lie ideal is proper, since it consists of diagonal matrices, while by Proposition~\ref{trace}, $[S,S]$ contains all matrices having trace zero, and since $d \geq 2$, some such matrices must be non-diagonal. Hence $[S,S]$ is not simple.
\end{proof}

\begin{proposition} \label{simple}
Let be $R$ a simple unital ring, $d \geq 2$ an integer, and $S=\M_d(R)$. Then the Lie ring $\, [S,S]$ is simple if and only if the following conditions hold:

$(1)$   $1 \not\in [R,R]$,

$(2)$  $\mathrm{char}(R)$ does not divide $d$.
\end{proposition}

\begin{proof}
Suppose that $[S,S]$ is simple as a Lie ring. By Lemma~\ref{intersection}, we have $Z(R) \cap [R,R]=0$, and hence (1) holds. Now, suppose that $\mathrm{char}(R)$ divides $d$. Then $I$ (the identity) is a nonzero matrix in $Z(S)$ with $\mathrm{trace}(I)=0$. By Proposition~\ref{trace}, $I \in [S,S]$, and hence the additive subgroup generated by $I$ is a nonzero Lie ideal of $[S,S]$, which is proper (as in the proof of Lemma~\ref{intersection}, this ideal consists of diagonal matrices, whereas $[S,S]$ does not), contradicting the simplicity of $[S,S]$. Thus, if $[S,S]$ is simple, then (1) and (2) must hold.

For the converse, suppose that (1) and (2) hold. It is well-known that $Z(R)$ is a field for any simple unital ring $R$.   We first note that it could not be the case that $2 = \mathrm{char}(S) \ (= \mathrm{char}(R))$, $Z(S)$ is a field, and $S$ has dimension $4$ over $Z(S)$. For in that case, since $d\geq 2$, we necessarily have $d=2$ and $R=Z(R)=Z(S)$ (where $R$ is identified with its diagonal embedding in $S$). But, this would violate (2). Thus, by Theorem~\ref{Herst}, given a proper Lie ideal $U \subseteq [S,S]$, we have $U \subseteq Z(S) = Z(R)$. Now, let $A \in U$ be any matrix. Since, $A \in Z(S)$, we have $A = \mathrm{diag}(a)$ for some $a\in Z(R)$. By Proposition~\ref{trace}, $\mathrm{trace}(A) = da \in [R,R] \cap Z(R)$, which, by (2), implies that $a \in [R,R] \cap Z(R)$ (since $d$ is a nonzero  element of the field $Z(R)$). By (1), this can only happen if $a=0$. Hence $A=0$, and therefore also $U=0$, showing that $[S,S]$ contains no nontrivial ideals. It remains only to show that $[ [S,S], [S,S]]\neq 0$. But, by Proposition~\ref{trace}, the matrix units $e_{12}$ and $e_{21}$ are elements of $[S,S]$, and hence $0\neq e_{11}-e_{22}=[e_{12},e_{21}] \in [ [S,S], [S,S]]$.
\end{proof}


\section{Commutators in Leavitt path algebras}\label{CommutatorsSection}

We now take up the first of our two main goals: to describe various elements of a Leavitt path algebra $L_K(E)$ which may be written as sums of commutators.   The main result of this section is Theorem \ref{LPA-comm}, where we give (among other things) necessary and sufficient conditions for the specific element $1_{L_K(E)}$ to be so written.  

We start by defining the relevant algebraic and graph-theoretic structures.  (See e.g. \cite{D} for additional information about some of the graph-theoretic terms used here.)  
A \emph{directed graph} $E=(E^0,E^1,r,s)$ consists of two  sets $E^0,E^1$, together with functions $s,r:E^1 \to E^0$, called {\it source} and {\it range}, respectively.    The word \emph{graph} will always mean \emph{directed graph}.     The elements of $E^0$ are called \emph{vertices} and the elements of $E^1$ \emph{edges}. The sets $E^0$ and $E^1$ are allowed to be of arbitrary cardinality.  We emphasize that loops (i.e., edges $e$ for which $s(e)=r(e)$),   and parallel edges (i.e., edges $f\neq g$ for which $s(f)=s(g)$ and $r(f)=r(g)$),  are allowed.   A \emph{path} $\mu$ in  $E$ is a finite sequence of (not necessarily distinct) edges
$\mu=e_1\dots e_n$ such that $r(e_i)=s(e_{i+1})$ for $i=1,\dots,n-1$; in this case, $s(\mu):=s(e_1)$ is the
\emph{source} of $\mu$, $r(\mu):=r(e_n)$ is the \emph{range} of $\mu$, and $n$ is the \emph{length} of $\mu$. We view the elements of $E^0$ as paths of length $0$.  We denote by ${\rm Path}(E)$ the set of all paths in $E$ (including paths of length $0$). 
A sequence $\{e_i\}_{i\in \N}$ of edges in $E$ is called an {\it infinite
path} in case $r(e_i)=s(e_{i+1})$ for all $i \in \N$.

 If $\mu = e_1\dots e_n$ is a path in $E$, and if $v=s(\mu)=r(\mu)$ and $s(e_i)\neq s(e_j)$ for every $i\neq j$, then $\mu$ is called a
\emph{cycle based at} $v$.   A cycle consisting of one edge is called a \emph{loop}.    A graph which contains no cycles is called \emph{acyclic}.   We note that if $\mu = e_1 \dots e_n$ is a cycle, then the sequence $e_1, ..., e_n, e_1, ..., e_n, e_1, ...$ is an infinite path.  

  A   vertex $v$  for which  the set $\{e\in E^1 \mid s(e)=v\}$ is finite is said to have {\it finite out-degree}; the graph $E$ is said to have {\it finite out-degree}, or  is said to be {\it row-finite}, if every vertex of $E$ has finite out-degree.     A graph for which both $E^0$ and $E^1$ are finite sets is called a \emph{finite} graph.  A vertex $v$ which is the source vertex of no edges of $E$ is called a \emph{sink}, while a vertex $v$ having finite out-degree which is not a sink is called a {\it regular} vertex.  
An edge $e$ is an {\it exit} for a path $\mu = e_1 \dots e_n$ if there exists $i$ ($1\leq i \leq n$)  such that $s(e)=s(e_i)$  and $e\neq e_i$.  We say that a vertex $v$ \emph{connects} to a vertex $w$ in case there is a path $p\in {\rm Path}(E)$ for which $s(p)=v$ and $r(p)=w$. 

Of central focus in this article are Leavitt path algebras.

\begin{definition}\label{definition}  
{\rm Let $K$ be a field, and let $E$ be a  graph. The {\em Leavitt path $K$-algebra} $L_K(E)$ {\em of $E$ with coefficients in $K$} is  the $K$-algebra generated by a set $\{v\mid v\in E^0\}$, together with a set of variables $\{e,e^*\mid e\in E^1\}$, which satisfy the following relations:

(V)  $vw = \delta_{v,w}v$ for all $v,w\in E^0$ \  (i.e., $\{v\mid v\in E^0\}$ is a set of orthogonal idempotents),

  (E1) $s(e)e=er(e)=e$ for all $e\in E^1$,

(E2) $r(e)e^*=e^*s(e)=e^*$ for all $e\in E^1$,

 (CK1) $e^*e'=\delta _{e,e'}r(e)$ for all $e,e'\in E^1$,

(CK2)  $v=\sum _{\{ e\in E^1\mid s(e)=v \}}ee^*$ for every regular vertex $v\in E^0$.  \hfill $\Box$
}

\end{definition}

\medskip


 We let $r(e^*)$ denote $s(e)$, and we let $s(e^*)$ denote $r(e)$.
If $\mu = e_1 \dots e_n \in {\rm Path}(E)$, then we denote by $\mu^*$ the element $e_n^* \dots e_1^*$ of $L_K(E)$. An expression of this form is called a \emph{ghost path}.


Many well-known algebras arise as the Leavitt path algebra of a
 graph.
For example, the classical Leavitt $K$-algebra $L_K(n)$ for $n\ge 2$; the full $d\times d$ matrix algebra ${\M}_d(K)$ over $K$; and the Laurent polynomial
algebra $K[x,x^{-1}]$ arise, respectively, as the Leavitt path algebra of the ``rose with $n$ petals" graph $R_n$ ($n\geq 2$); the oriented line graph $A_d$ having $d$ vertices; and the ``one vertex, one loop" graph $R_1$ pictured here.

$$R_n \ =  \xymatrix{ & {\bullet^v} \ar@(ur,dr) ^{e_1} \ar@(u,r) ^{e_2}
\ar@(ul,ur) ^{e_3} \ar@{.} @(l,u) \ar@{.} @(dr,dl) \ar@(r,d) ^{e_n}
\ar@{}[l] ^{\ldots} } \ \  \ \ \ \  \ A_d \ = \ \xymatrix{{\bullet}^{v_1} \ar [r] ^{e_1} & {\bullet}^{v_2}  \ar@{.}[r] & {\bullet}^{v_{d-1}} \ar [r]
^{e_{d-1}} & {\bullet}^{v_d}} \ \ \ \ \  \ R_1 \ = \   \xymatrix{{\bullet}^{v} \ar@(ur,dr) ^x}$$

\medskip




Although various bases for an algebra of the form $L_K(E)$ can be identified, such bases typically don't lend themselves well to defining $K$-linear transformations from $L_K(E)$ to other $K$-spaces.   However, $L_K(E)$ may be viewed as a quotient of the {\it Cohn path algebra} $C_K(E)$  by a suitable ideal $N$  (defined immediately below).  The advantage of this point of view in the current discussion is that the Cohn path algebra possesses an easily described basis, and the ideal $N$ of $C_K(E)$ behaves well vis-a-vis a specific $K$-linear transformation defined on $C_K(E)$ (see Lemma \ref{trace1}).    The Cohn path algebra is a generalization to all graphs of the algebras $U_{1,n}$ defined by P.\ M.\ Cohn in \cite{Cohn}.

\begin{definition}\label{Cohndefinition}  
{\rm Let $K$ be a field, and let $E$ be a  graph. The {\em Cohn path $K$-algebra} $C_K(E)$ {\em of $E$ with coefficients in $K$} is  the $K$-algebra generated by a set $\{v\mid v\in E^0\}$, together with a set of variables $\{e,e^*\mid e\in E^1\}$, which satisfy the relations (V), (E1), (E2), and (CK1) of Definition~\ref{definition}.  

We let $N \subseteq C_K(E)$ denote the ideal of $C_K(E)$ generated by elements of the form $v~-~\sum_{\{e\in E^1 \mid s(e)=v\}} ee^*,$ where $v \in E^0$ is a regular vertex. \hfill $\Box$
}
\end{definition}

In particular, we may view the Leavitt path algebra $L_K(E)$   as the quotient  algebra 
$$L_K(E) \cong C_K(E)/N.$$ 
If $E$ is a graph for which $E^0$ is finite, then $\sum _{v\in E^0} v$ is the multiplicative identity, viewed either as an element of   $L_K(E)$ or $C_K(E)$.  If $E^0$ is infinite, then  both $L_K(E)$ and $C_K(E)$ are   nonunital.  
Identifying $v$ with $v^*$ for each $v \in E^0$, one can show that $$\{pq^* \mid p,q \in {\rm Path}(E) \text{ such that } r(p)=r(q)\}$$ is a basis for $C_K(E)$.

\begin{lemma}\label{annihilators}
Let $v\in E^0$ be a regular vertex, and let $y$ denote the element
$ v - \sum_{\{e\in E^1 \mid s(e)=v\}}ee^*$  of the ideal $N$ of $C_K(E)$ described in Definition \ref{Cohndefinition}.  

$(1)$  If $p \in {\rm Path}(E) \setminus E^0$, then $yp=0$.

$(2)$  If $q\in {\rm Path}(E) \setminus E^0$, then $q^*y=0$.

\end{lemma}

\begin{proof}
(1)  Write $p=fp'$ for some $f \in E^1$ and $p' \in {\rm Path}(E)$.   If $s(f) \neq v$ then $yp=0$ immediately.  On the other hand, if $s(f)=v$ then $f \in \{e\in E^1 \mid s(e)=v\}$, in which case, by (CK1), we get $$yp = (v-\sum_{\{e\in E^1 \mid s(e)=v\}} ee^*)fp' = fp' - ff^*fp' = fp' - fp' = 0.$$  
The proof of (2) is similar.
\end{proof}

\begin{definition}\label{TDefinition}
{\rm Let $K$ be a field, and let $E$ be a   graph.  We index the vertex set $E^0$ of $E$ by a set $I$,  and write $E^0 = \{v_i \mid i\in I\}$.  Let $K^{(I)}$ denote the direct sum of copies of $K$ indexed by $I$.  For each $i \in I$, let $\epsilon_i \in K^{(I)}$ denote the element with $1 \in K$ as the $i$-th coordinate and zeros elsewhere.   Let $T : C_K(E) \rightarrow K^{(I)}$ be the $K$-linear map which acts as
$$T(pq^*) = \left\{ \begin{array}{cl}
\epsilon_i & \textrm{if } q^*p = v_i\\
0 & \textrm{otherwise}
\end{array}\right.$$ 
on the aforementioned basis of $C_K(E)$. \hfill $\Box$
}
\end{definition}

We note that $T(v_i) = \epsilon_i$ for all $i\in I$, and for any $p \in {\rm Path}(E) \setminus E^0$, $T(p) = 0 = T(p^*)$. 

\begin{lemma} \label{trace2}
Let $K$ be a field,  let $E$ be    graph, and write  $E^0 = \{v_i \ | \ i\in I\}$.  Let $T$ denote the $K$-linear transformation given in Definition~\ref{TDefinition}.   Then for all $x,y \in C_K(E)$ we have $T(xy) = T(yx)$.  In particular, $T(z)=0$ for every $z\in [C_K(E),C_K(E)]$.  
\end{lemma}

\begin{proof}
Since $T$ is $K$-linear, it is enough to establish the result for $x$ and $y$ that are elements of the basis for $C_K(E)$ described above. That is, we may assume that $x=pq^*$ and $y=tz^*$, for some $p,q,t,z \in {\rm Path}(E)$ with $r(p)=r(q)=v_i\in E^0$ and $r(t)=r(z)=v_j\in E^0$. Now, $pq^*tz^* = 0$ unless either $t=qh$ or $q=th$ for some $h \in {\rm Path}(E)$. Also, $tz^*pq^* = 0$ unless either $p=zg$ or $z=pg$ for some $g \in {\rm Path}(E)$. Let us consider the various resulting cases separately.

Suppose that $t=qh$ for some $h \in {\rm Path}(E)$ but $z \neq pg$ for all $g \in {\rm Path}(E)$. Then $pq^*tz^* = pq^*qhz^* = phz^*$ and $T(pq^*tz^*) = T(phz^*) = 0$, since $z \neq ph$. Also, as mentioned above, $tz^*pq^* = 0$ unless $p=zg$ for some $g \in {\rm Path}(E)$. If $tz^*pq^* = 0$, then we have $T(tz^*pq^*) = 0 = T(pq^*tz^*)$. Therefore, let us suppose that $p=zg$ for some $g \in {\rm Path}(E)$. Then $tz^*pq^* = tz^*zgq^* = tgq^* = qhgq^*$, and hence $T(tz^*pq^*) = T(qhgq^*) = 0$ unless $hg \in E^0$. But, $hg \in E^0$ can happen only if $h=g \in E^0$, in which case $p=z$ (since $p \neq 0$), contradicting our assumption. Therefore, $p\neq zg$ for all $g \in {\rm Path}(E)$, and we have $T(pq^*tz^*) = 0 = T(tz^*pq^*)$.

Let us next suppose that $t=qh$ and $z = pg$ for some $g,h \in {\rm Path}(E)$. Then $pq^*tz^* = pq^*qhg^*p^* = phg^*p^*$, and hence $T(pq^*tz^*) = \epsilon_j$ if $g = h$ and $0$ otherwise. Also, $tz^*pq^* = qhg^*p^*pq^* = qhg^*q^*$, and so $T(tz^*pq^*) = \epsilon_j$ if $g = h$ and $0$ otherwise. Thus, in either case we have $T(pq^*tz^*) = T(tz^*pq^*)$.

Now suppose that $t \neq qh$ for all $h \in {\rm Path}(E)$ but $z = pg$ for some $g \in {\rm Path}(E)$. Then $pq^*tz^* = pq^*tg^*p^* \neq 0$ only if $q=th$ for some $h \in {\rm Path}(E)$. Hence $T(pq^*tz^*) \neq 0$ only if $z=p$ and $q=t$, which is not the case, by hypothesis. Similarly, $tz^*pq^* = tg^*p^*pq^* = tg^*q^*$, and hence $T(tz^*pq^*) \neq 0$ only if $t=qg$, which is not the case. Thus, $T(tz^*pq^*) = 0 = T(pq^*tz^*)$.

Finally, suppose that $t \neq qh$ and $z \neq pg$ for all $g,h \in {\rm Path}(E)$. Then $pq^*tz^* = 0$ unless $q=th$ for some $h\in {\rm Path}(E)$, and $T(pq^*tz^*) = 0$ unless $q=th$ and $p=zh$ for some $h\in {\rm Path}(E)$. Similarly, $T(tz^*pq^*) = 0$ unless $q=th$ and $p=zh$ for some $h\in {\rm Path}(E)$. Thus, let us suppose that $q=th$ and $p=zh$ for some $h\in {\rm Path}(E)$. In this case, $$T(pq^*tz^*) = T(zhh^*t^*tz^*) = T(zhh^*z^*) = v_i = T(thh^*t^*) = T(tz^*zhh^*t^*) = T(tz^*pq^*),$$ as desired.

Therefore, in all cases $T(pq^*tz^*) = T(tz^*pq^*)$, proving the first claim of the lemma. The second follows trivially.
\end{proof}

\begin{definition} \label{number-def}
{\rm Let $E$ be a graph, and write $E^0 = \{v_i \mid i \in I\}$.  If  $v_i$ is a regular vertex, for all $j \in I$ let $a_{ij}$ denote the number of edges $e \in E^1$ such that $s(e) = v_i$ and $r(e) = v_j$. In this situation,  define $$B_i = (a_{ij})_{j \in I} - \epsilon_i \in \Z^{(I)}.$$   On the other hand, let $$B_i = (0)_{j \in I} \in \Z^{(I)},$$ if $v_i$ is not a regular vertex.  \hfill $\Box$
}
\end{definition}

\begin{lemma} \label{trace1}
Let $K$ be a field, let $E$ be a  graph, and write $E^0 = \{v_i \mid i \in I\}$. Then for all elements $w$ of the ideal $N$ of $C_K(E)$ we have $T(w) \in {\rm span}_K \{B_i \mid i \in I\} \subseteq K^{(I)}$.
\end{lemma}

\begin{proof}
It is sufficient to show that for any generator $$y_i = v_i-\sum_{\{e\in E^1 \mid s(e)=v_i\}} ee^*$$ of $N$ and any two elements $c,c'$ of  $C_K(E)$, we have 
 $T(cy_ic') \in \mathrm{span}_K\{B_i \mid i \in I\} \subseteq K^{(I)}$. But, by Lemma~\ref{trace2}, $T(cy_ic') = T(c'cy_i)$, and hence we only need to show that $T(cy_i) \in \mathrm{span}_K\{B_i \mid i \in I\}$ for any $c \in C_K(E)$. Further, since $T$ is $K$-linear, we may assume that $c=pq^*$ belongs to the basis for $C_K(E)$ described above; in particular,  $p,q \in {\rm Path}(E)$.    Again using Lemma~\ref{trace2}, we have $T(cy_i)=T(pq^*y_i)=T(q^*y_ip)$.  But, by Lemma \ref{annihilators}, the expression $q^*y_ip$ is zero unless $q^* = v_i = p$.  So the only nonzero term of the form $T(cy_i)$ is 
$$T(cy_i) = T(y_i) = T( v_i - \sum_{\{e\in E^1 \mid s(e)=v_i\}} ee^*) = \epsilon_i - (a_{ij})_{j \in I} = -B_i,$$ since for each $e \in E^1$ with $s(e)=v_i$ and $r(e) = v_j$, we have $T(ee^*) = \epsilon_j$. Clearly $-B_i \in {\rm span}_K \{B_i \mid i \in I\}$, and we are done.  
\end{proof}

Here now is our first goal, achieved.

\begin{theorem} \label{LPA-comm}
Let $K$ be a field, let $E$ be a graph, and write $E^0 = \{v_i \mid i \in I\}$. For each $i \in I$ let $B_i$ denote the element of $K^{(I)}$ given in Definition~\ref{number-def}, and let $\, \{k_i \mid i\in I\} \subseteq  K$ be a set of scalars where  $k_i = 0$ for all but finitely many $i \in I$.  Then 
$$\sum_{i\in I}k_iv_i \in [L_K(E), L_K(E)] \mbox{ if and only if } \, (k_i)_{i\in I} \in \mathrm{span}_K\{B_i \mid i \in I\}.$$ In particular, if $E^0$ is finite $($so that  $L_K(E)$ is unital$)$, then $$1_{L_K(E)} \in [L_K(E), L_K(E)]  \mbox{ if and only if } \, (1, \dots, 1) \in \mathrm{span}_K\{B_i \mid i \in I\} \subseteq K^{(I)}.$$
\end{theorem}

\begin{proof}
First, suppose that $(k_i)_{i\in I} \in \mathrm{span}_K\{B_i \mid i \in I\}.$ For all $i,j \in I$ such that $v_i$ is regular, let $e^{ij}_1, \dots, e^{ij}_{a_{ij}}$ be all the edges $e \in E^1$ satisfying $s(e) = v_i$ and $r(e) = v_j$. (We note that there are only finitely many such elements.) Then for each regular $v_i$ we have
 $$\sum_{j\in I}\sum_{l=1}^{a_{ij}} [e^{ij}_l, (e^{ij}_l)^*] = \sum_{j\in I}\sum_{l=1}^{a_{ij}} e^{ij}_l(e^{ij}_l)^* - \sum_{j\in I}\sum_{l=1}^{a_{ij}} (e^{ij}_l)^*e^{ij}_l $$
 $$ = 
 \sum_{\{e\in E^1 \mid s(e)=v_i\}} ee^* - \sum_{j\in I}\sum_{l=1}^{a_{ij}} (e^{ij}_l)^*e^{ij}_l  = v_i - \sum_{j\in I} a_{ij}v_j.$$

By hypothesis we can  write $(k_i)_{i\in I} = \sum_{i\in I} t_i B_i$ for some $t_i\in K$ (all but finitely many of which are $0$). We may assume that $t_i=0$ whenever $v_i$ is not regular, since in that case $B_i$ is zero. Thus, $$\sum_{i\in I}k_iv_i = -\sum_{i \in I}t_i(v_i - \sum_{j\in I} a_{ij}v_j),$$
which is an element of $[L_K(E), L_K(E)]$, by the above computation.

For the converse, viewing $L_K(E)$ as $C_K(E)/N$, we shall show that if $\sum_{i\in I}k_iv_i + N  \in [L_K(E), L_K(E)]$ for $v_i \in E^0$ and $k_i \in K$ satisfying the hypotheses in the statement, then $(k_i)_{i\in I} \in \mathrm{span}_K\{B_i \mid i \in I\}$. Now, if $\sum_{i\in I}k_iv_i + N  \in [L_K(E), L_K(E)]$, then there are elements $x_j,y_j \in C_K(E)$ such that $ \sum_{i\in I}k_iv_i = \sum_j [x_j,y_j]  + w$ for some $w \in N$. Using Lemma~\ref{trace2} we get 
$$(k_i)_{i\in I} = T(\sum_{i\in I}k_iv_i) =T( \sum_j [x_j,y_j] ) + T(w) = 0 + T(w) = T(w). $$
 Lemma~\ref{trace1} then   gives the desired result.

To prove the final claim, write $E^0 = \{v_1, \dots, v_m\}$ and use the previously noted fact that   $1_{L_K(E)} = v_1 + \dots + v_m$.
\end{proof}

It will be useful to identify various additional elements of $[L_K(E),L_K(E)]$.

\begin{lemma}\label{commutatorlemma}
Let $K$ be a field, $E$ a  graph, and $p,q \in {\rm Path}(E) \setminus E^0$ any paths.

$(1)$ If $s(p) \neq r(p)$, then $p,p^* \in  [L_K(E), L_K(E)]$.

$(2)$ If $p\neq qx$ and $q \neq px$ for all $x \in {\rm Path}(E)$ with $s(x) = r(x)$, then $pq^* \in  [L_K(E), L_K(E)]$.

\end{lemma}

\begin{proof}
(1) If $s(p) \neq r(p)$, then $r(p)p=0=p^*r(p)$, and hence $p = [p,r(p)]$ and $p^* = [r(p),p^*]$.

(2) We have $[p,q^*]=pq^*-q^*p$. If $p\neq qx$ and $q \neq px$ for all $x \in {\rm Path}(E)$, then $q^*p=0$, and hence $pq^* \in  [L_K(E), L_K(E)]$. Let us therefore suppose that either $p=qx$ or $q=px$ for some $x \in {\rm Path}(E)$ such that $s(x) \neq r(x)$. Thus $[p,q^*]=pq^*-x$ in the first case, and $[p,q^*]=pq^*-x^*$ in the second. In either situation, (1) implies that $pq^* \in  [L_K(E), L_K(E)]$.
\end{proof}

Theorem \ref{LPA-comm} gives necessary and sufficient conditions which ensure that elements of $L_K(E)$ having a specific form (namely, $K$-linear combinations of vertices) lie in $[L_K(E),L_K(E)]$.  This result suffices to meet our needs in this article, to wit, to help establish Theorem \ref{LPA-simple} below.   Using the results of this section, the second author has  generalized Theorem \ref{LPA-comm}  by providing  necessary and sufficient conditions for an arbitrary element of $L_K(E)$ to lie in $[L_K(E),L_K(E)]$;  see \cite[Theorem 15]{ZM2}.


\section{Simple Leavitt path algebras and associated Lie \\ algebras}\label{SimpleLPAandbracketsection}

In this section we apply the results proved in Section~\ref{CommutatorsSection} together with Herstein's result (Theorem~\ref{Herst}) in order to achieve our second main goal, namely, to identify the fields $K$ and row-finite graphs  $E$ for which the simple Leavitt path algebra $L_K(E)$ yields a  simple Lie algebra $[L_K(E),L_K(E)]$.   We begin by recording two basic facts about Leavitt path algebras.

\begin{lemma}\label{commsimpleisKlemma}

$(1)$ There is up to isomorphism exactly one simple commutative Leavitt path algebra, specifically the algebra  $K\cong L_K(\bullet)$. 

$(2)$ The only $K$-division algebra  of the form $L_K(E)$ for some graph $E$ is $K \cong L_K(\bullet)$.
\end{lemma}

 \begin{proof}
(1) This follows from \cite[Proposition 2.7]{AC}; we give an alternate  proof here for completeness.   Suppose $E$ is a graph other than \ $\bullet$ \ for which $L_K(E)$ is simple and commutative.    If $E$ were to contain no edges, then $E$ would consist of (at least two) isolated vertices, and thus $L_K(E)$ would not be   simple.  So we may assume that $E$ contains at least one edge.  If $E$ contains an edge $e$ for which $s(e)\neq r(e)$,  then, by Lemma~\ref{commutatorlemma}(1), we have $0\neq e\in [L_K(E),L_K(E)]$. On the other hand, if $E$ contains no such edges, then all edges in $E$ are loops.  In this situation,   by the simplicity of $L_K(E)$, there can only be one vertex  $v$ in $E$. It could not be the case that there is exactly one loop based at $v$, since then $L_K(E)\cong K[x,x^{-1}]$, which is not simple. So let $p,q$ be two distinct loops based at $v$. Then, by Lemma~\ref{commutatorlemma}(2), we have $0\neq pq^*\in [L_K(E),L_K(E)]$, completing the proof.  
  
  (2)  Since a division algebra has no zero divisors, in order for $L_K(E)$ to be such a ring, the graph $E$ must have exactly one vertex and at most one loop at that vertex.   But as noted previously, the Leavitt path algebra of the graph with one vertex and one loop is isomorphic to $K[x,x^{-1}]$, and thus is not a division ring.   The result follows.
\end{proof} 

Indeed, the proof of Lemma \ref{commsimpleisKlemma}(1) establishes that for any field $K$ and graph $E$, if $L_K(E) \cong K$, then $E$ must be the trivial graph   $\bullet $. (The converse is obvious.)
   Accordingly, we  call a simple Leavitt path algebra $L_K(E)$ {\it nontrivial} in case $L_K(E) \not\cong K$.

\begin{lemma}\label{double}
Let $K$ be a field, $E$ a  graph, and $R=L_K(E)$ a Leavitt path algebra. If $\, [R, R] \neq 0$, then  $\, [[R,R],[R,R]]\neq 0$. In particular, if $R$ is a nontrivial simple Leavitt path algebra, then $\, [[R,R],[R,R]]\neq 0.$
\end{lemma}

\begin{proof}
First, suppose there is an edge $e\in E^1$ that is not a loop. Then $r(e)\neq s(e)$, implying that $e^*r(e)=0$ and $r(e)e=0$.  Thus
$$ [[r(e),e^*], [e,r(e)]]  = [e^*,e] = r(e)-ee^* \in [[R,R],[R,R]]$$ is  nonzero, since $(r(e)-ee^*)r(e)=r(e)\neq 0$.  Next, suppose that $v$ is a vertex at which two distinct loops $e$ and $f$ are based.   Then $$[[e,e^*],[e,f]]=[ee^*-v,ef-fe]=ef-efee^*+fe^2e^* \in [[R,R],[R,R]]$$ is  nonzero,  since multiplying this element on the left by $f^*$ and on the right by $e$ yields the nonzero element $e^2$.  Thus the only remaining  configuration for $E$ not covered by these two cases 
 is that $E$ is a disjoint union of isolated vertices together with vertices at which there is exactly one loop.  But in this case  $L_K(E)$ is a direct sum of copies of $K$ with copies of $K[x,x^{-1}]$, so is commutative, and hence $[R,R] = 0$.

The second statement follows immediately from Lemma~\ref{commsimpleisKlemma}(1).  
\end{proof}

We note that the first statement of Lemma~\ref{double}  does not hold for an arbitrary ring $R$. For instance, let $R$ be the associative (unital or otherwise) ring generated by the following generators and relations $$\langle x,y : x^3=y^3=xy^2=yx^2=x^2y=y^2x=xyx=yxy=0\rangle.$$ Then $[x,y] \neq 0$, and hence $[R,R] \neq 0$. But, all the nonzero commutators in $R$ are integer multiples of $[x,y]$, and hence $[[R,R],[R,R]]=0$.

A description of the  row-finite graphs $E$  and fields $K$ for which $L_K(E)$ is simple is given in \cite[Theorem 3.11]{AAP}. Using  \cite[Lemma 2.8]{Exchange} to streamline the statement of this result, we have  

\begin{theorem}[The Simplicity Theorem] \label{SimplicityTheorem}  
Let $K$ be a field, and let  $E$ be a  row-finite graph.  Then $L_K(E)$ is simple if and only if $E$ has the following two properties.

$(1)$  Every vertex $v$ of $E$ connects to every sink and every infinite path of $E$.

$(2)$   Every cycle of $E$ has an exit.

\noindent In particular, if $E$ is finite, then $L_K(E)$ is simple if and only if
every vertex $v$ of $E$ connects to every sink and every cycle of $E$, and
every cycle of $E$ has an exit.
\end{theorem}

  Specifically, we note that the simplicity of the algebra $L_K(E)$ is independent of $K$.  (This is intriguing, especially in light of the fact that we will show below that the simplicity of the corresponding  Lie algebra $[L_K(E),L_K(E)]$ does indeed depend on $K$.)

  \begin{example}\label{simpleexample}
{\rm   Let $E$ be the graph pictured here.   
$$ \  \ \ \  
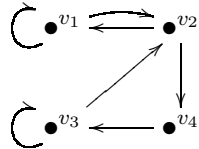
\xymatrix{{\bullet}^{v_1} 
\ar@(dl,ul) 
\ar@/^.5pc/[r]   &
 {\bullet}^{v_2} 
\ar[l]
 \ar[d]
  \\
  {\bullet}^{v_3} 
  \ar@(dl,ul) 
  \ar[ur] 
    & {\bullet}^{v_4} 
  \ar[l]
 }$$
 \smallskip
 \noindent
By applying Theorem~\ref{SimplicityTheorem}, we conclude that $L_K(E)$ is simple for any field $K$.   \hfill $\Box$

}
 \end{example}

  The following is due to Aranda Pino and Crow.

\begin{theorem}[Theorem 4.2 from \cite{AC}]\label{center}
Let $K$ be a field, and let $E$ be a row-finite graph for which $L_K(E)$ is a simple Leavitt path algebra. 

$(1)$ If $L_K(E)$ is unital, then $Z(L_K(E)) = K$.

$(2)$  If $L_K(E)$ is not unital, then $Z(L_K(E)) = 0$.

\end{theorem}

This result immediately allows us to identify simple Lie algebras arising from graphs having infinitely many vertices.

\begin{corollary}\label{LPA-simplenonunital}
Let $K$ be a field, and let $E$ be a row-finite  graph  having infinitely many vertices, for which  $L_K(E)$ is a simple Leavitt path algebra. Then $\, [L_K(E),L_K(E)]$ is a simple Lie $K$-algebra.
\end{corollary}

\begin{proof}
This follows by combining Theorem~\ref{center}(2) with Corollary~\ref{Herst-Cor} and Lemma~\ref{double}, since if $E$ has infinitely many vertices, then $L_K(E)$ is not unital.
\end{proof}

On the other hand, we get the following result for graphs having finitely many vertices.  

\begin{corollary} \label{Leavitt-1}
Let $K$ be a field, and let $E$ be a finite graph for which $L_K(E)$ is a nontrivial simple Leavitt path algebra. Then the Lie $K$-algebra $\, [L_K(E),L_K(E)]$ is simple  if and only if $\, 1 = 1_{L_K(E)} \notin [L_K(E),L_K(E)]$.
\end{corollary}

\begin{proof}
If $1 \in [L_K(E),L_K(E)]$, then the  $K$-subspace $\langle 1\rangle$ of $[L_K(E),L_K(E)]$ generated by $1$   is a nonzero Lie ideal of $[L_K(E),L_K(E)]$. Since $\langle 1\rangle$ is a commutative subalgebra of $L_K(E)$, by Lemma \ref{double} we have that $\langle 1\rangle$ is proper. Thus, $[L_K(E),L_K(E)]$ is not simple. 

Conversely, if $1 \notin [L_K(E),L_K(E)]$, then we have $Z(L_K(E))\cap [L_K(E),L_K(E)]=0$, by Theorem~\ref{center}(1). Since $L_K(E)$ is nontrivial simple,
$[[L_K(E),L_K(E)],[L_K(E),L_K(E)]]\neq 0$, by Lemma~\ref{double}. Further, it cannot
be the case that char$(K) = 2$ and $L_K(E)$ is 4-dimensional over $Z(L_K(E))
= K$, for then we would have $L_K(E) \cong \mathbb{M}_2(K)$.  (It is well-known that a
4-dimensional central simple K-algebra that is not a division ring must
be of this form, and by Lemma~\ref{commsimpleisKlemma}(2) $L_K(E)$ is not a division ring.) But, if char$(K) = 2$, then $1 \in [\mathbb{M}_2(K),\mathbb{M}_2(K)]$  by
Proposition~\ref{trace}, contradicting our assumption. Thus, the desired
conclusion now follows from Theorem~\ref{Herst}.
\end{proof}

Now combining  Theorem~\ref{LPA-comm} with Corollary~\ref{Leavitt-1}, we have achieved our second main goal.

\begin{theorem} \label{LPA-simple}
Let $K$ be a field, and let $E$ be a finite  graph for which $L_K(E)$ is a nontrivial simple Leavitt path algebra. Write $E^0 = \{v_1, \dots, v_m\}$, and for each $\, 1\leq i \leq m$ let $B_i$ be as in Definition~\ref{number-def}. Then the Lie $K$-algebra $\, [L_K(E),L_K(E)]$ is simple if and only if $\, (1,\dots, 1) \not\in \mathrm{span}_K\{B_1, \dots, B_m \}$.
\end{theorem}

Here is the first of many consequences of Theorem~\ref{LPA-simple}. 

\begin{corollary} \label{LPA-simple2}
Let $K$ be a field,  let $E$ be a finite  graph for which $L_K(E)$ is a nontrivial simple Leavitt path algebra, and let $d$ be a positive integer.   Write $E^0 = \{v_1, \dots, v_m\}$, and for each $\, 1\leq i \leq m$ let $B_i$ be as in Definition~\ref{number-def}. Then the Lie $K$-algebra $\, [\M_d(L_K(E)),\M_d(L_K(E))]$ is simple  if and only if $\, (1,\dots, 1) \not\in \mathrm{span}_K\{B_1, \dots, B_m \}$ and $\, \mathrm{char}(K)$ does not divide $d$.
\end{corollary}

\begin{proof}
The $d=1$ case is precisely Theorem~\ref{LPA-simple} (noting of course that ${\rm char}(K)$ never divides $1$), while the $d\geq 2$ case  follows by applying Proposition~\ref{simple} (and Lemma~\ref{simpleasLieringiffsimpleasLiealgebra}) to Theorem~\ref{LPA-comm}.
\end{proof}

Since for any positive integer $d$ and any  graph $E$, the $K$-algebra $\M_d(L_K(E))$ is isomorphic to a Leavitt path algebra with coefficients in $K$ (see e.g. \cite[Proposition 9.3]{AT}), the previous corollary can in fact be established using Theorem~\ref{LPA-simple} directly.  
In particular, we get as a consequence of Corollary~\ref{LPA-simple2} a second, more efficient, proof of  the aforementioned previously-established result for matrix rings over Leavitt algebras. 

\begin{corollary}[Theorem 3.4 from \cite{AF}]\label{MainCorformatricesoverLeavittalgebras}
Let $K$ be a field, let $n \geq 2$ and $d\geq 1$ be integers, and let $L_K(n)$ be the Leavitt $K$-algebra. Then the Lie $K$-algebra $\, [\M_d(L_K(n)),\M_d(L_K(n))]$ is simple  if and only if $\, \mathrm{char}(K)$ divides $n-1$ and does not divide $d$.
\end{corollary}

\begin{proof}
Let $E$ be the graph having one vertex $v_1$ and $n$ loops. Then $L_K(n) \cong L_K(E)$. We have $B_1 = n-1 \in K$, and hence $1 \not\in \mathrm{span}_K\{B_1\} = (n-1)K$ if and only if $\mathrm{char}(K)$ divides $n-1$. The result now follows from Corollary~\ref{LPA-simple2}.
\end{proof}

Throughout the remainder of the article, in a standard pictorial description of a directed graph $E$, a $(n)$ written on an edge connecting two vertices indicates that there are $n$ edges connecting those two vertices in $E$.   We now recall (the germane portion of) \cite[Lemma 5.1]{AALP}.

\begin{lemma}
\label{matricesoverLeavittalgebrasareleavittpathalgebras}
For integers $d\geq 2$ and $n\geq 2$ we denote by 
$E_n^d$ the following graph.
$$  
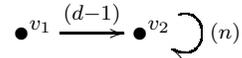
\xymatrix{
  \bullet^{v_1}  \ar[r]^{(d-1)}
 & \bullet^{v_2}  \ar@(ur,dr)[]^{(n)}  } \ \ \ \  \
$$
 Then for any field $K$, we have $
{\M}_d(L_K(n)) \cong L_K(E_n^d)$.
\end{lemma}

We now present a number of examples which highlight the computational nature of Theorem~\ref{LPA-simple}.     We start by offering an additional proof of the $d\geq 2$ case of
Corollary~\ref{MainCorformatricesoverLeavittalgebras},  one which  makes direct use of the Leavitt path algebra structure of $\M_d(L_K(n))$. 

\medskip

{\it Additional proof of the $d\geq 2$ case of Corollary~\ref{MainCorformatricesoverLeavittalgebras}:} By Lemma \ref{matricesoverLeavittalgebrasareleavittpathalgebras}, we have $L_K(E_n^{d}) \cong {\M}_d(L_K(n))$, and it is clear that the graph $E_n^d$  yields  $B_1 = (-1,d-1)$ and $B_2 = (0,n-1)$.  By Theorem~\ref{LPA-simple},  we seek properties of the integers $n,d$ and field $K$ for which $(1,1)\in {\rm span}_K\{B_1,B_2\}$, i.e., for which the equation $k_1(-1,d-1) + k_2(0,n-1) = (1,1)$ has solutions in $K\times K$. Equating coordinates, we seek to solve 
$$\begin{array}{lcl}  \ \  \ \ \ \ \ -k_1 & = &1 \\ (d-1)k_1 + (n-1)k_2& = & 1 \end{array}$$
with $k_1, k_2 \in K$.  So $k_1 = -1$, which  gives $-(d-1) + k_2(n-1) = 1$, and thus $d = k_2(n-1)$. In case $n-1\neq 0$ in $K$ (i.e., ${\rm char}(K)$ does not divide $n-1$), this obviously has a solution,  while in case $n-1=0$  in $K$, the equation has a solution precisely when $d=0$ in $K$, i.e., when ${\rm char}(K)$ divides $d$.  \hfill $\Box$

\begin{remark} \label{simplicityremark}
The following observations follow directly from Corollary~\ref{MainCorformatricesoverLeavittalgebras}.

(1)  The Lie $K$-algebra $[{\M}_d(L_K(n)), {\M}_d(L_K(n))]$ is not simple when ${\rm char}(K)=0$.

(2)  Let $\mathcal{P} = \{p_1,p_2,\dots,p_t\}$ be a finite set of primes, and let $q = p_1p_2\cdots p_t \in \N$.   Then the Lie $K$-algebra $[L_K(q+1),L_K(q+1)]$ is simple if and only if ${\rm char}(K)\in \mathcal{P}$.  

(3)  The Lie $K$-algebra  $[L_K(2), L_K(2)]$ is not simple for all fields $K$.    \hfill $\Box$ 
\end{remark}

The observations made in Remark~\ref{simplicityremark}  naturally suggest the following question:   are there graphs $E$ for which $[L_K(E), L_K(E)]$ is  a simple Lie $K$-algebra for {\it all} fields $K$?   Let us construct such an example now.  

\begin{example}\label{simplebracketbutpdoesntdivideindex}
{\rm  We revisit the graph $E$ presented in Example \ref{simpleexample}.    
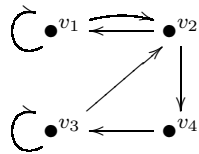
$$ \    \ \ \  \xymatrix{{\bullet}^{v_1} 
\ar@(dl,ul) 
\ar@/^.5pc/[r]   &
 {\bullet}^{v_2} 
\ar[l]
 \ar[d]
  \\
  {\bullet}^{v_3} 
  \ar@(dl,ul) 
  \ar[ur] 
    & {\bullet}^{v_4} 
  \ar[l]
 }$$
 \smallskip
 \noindent
Arising from this graph  we have   $B_1 = (0,1,0,0), B_2 = (1,-1,0,1)$, $B_3 = (0,1,0,0)$, and $B_4=(0,0,1,-1)$.   Let us determine whether or not $(1,1,1,1)$ is in ${\rm span}_K\{B_1,B_2,B_3,B_4\}$.  Upon building the appropriate augmented matrix of the resulting linear system, and using one row-swap and two add-an-integer-multiple-of-one-row-to-another operations, we are led to the matrix
$$ \begin{pmatrix}
1 & -1 & 1 &0 & \vdots & 1 \\
0 & 1 & 0 &0 & \vdots & 1 \\
0 & 0 & 0 & 1 & \vdots & 1 \\
0 & 0 & 0 & 0 & \vdots &  1 \\
\end{pmatrix} \ .
$$
The final row indicates that the system has no solutions,  regardless of the characteristic of $K$.    So, by Theorem~\ref{LPA-simple}, the Lie  algebra $[L_{K}(E), L_{K}(E)]$ is simple for any field $K$. \hfill $\Box$
}
\end{example}

In particular, Example~\ref{simplebracketbutpdoesntdivideindex} together with Remark~\ref{simplicityremark}(1)   show that Theorem~\ref{LPA-simple} indeed  
 enlarges the previously-known class of Leavitt path algebras for which the associated  Lie algebra is simple.

We consider a complementary question arising from Remark~\ref{simplicityremark}(2).  Specifically, for a given set of primes we produce a graph for which the  Lie algebras corresponding to the associated Leavitt path algebras over specified fields are {\it not} simple.

\begin{example}\label{graphswithprescribedprimesnotsimple}
{\rm
 Let $\mathcal{P} = \{p_1,p_2,\dots,p_t\}$ be a finite set of primes, let $q = p_1p_2\cdots p_t \in \N$, and let $E_q$ be the graph pictured below.   
$$ \  \ \ \  
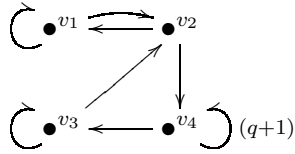
\xymatrix{{\bullet}^{v_1} 
\ar@(dl,ul) 
\ar@/^.5pc/[r]   &
 {\bullet}^{v_2} 
\ar[l]
 \ar[d]
  \\
  {\bullet}^{v_3} 
  \ar@(dl,ul) 
  \ar[ur] 
    & {\bullet}^{v_4} 
  \ar[l]
  \ar@(dr,ur)_{(q+1)}
 }$$
By Theorem~\ref{SimplicityTheorem},  we see that $L_K(E_q)$ is simple for any integer $q$ and any field $K$. 

For this graph $E_q$ we have $B_1 = (0,1,0,0), B_2 = (1,-1,0,1), B_3 = (0,1,0,0)$, and $B_4=(0,0,1,q)$.  Let us determine whether or not $(1,1,1,1)$ is in ${\rm span}_K\{B_1,B_2,B_3,B_4\}$.  Upon building the appropriate augmented matrix of the resulting linear system,  and using a sequence of row-operations analogous to the one used in Example~\ref{simplebracketbutpdoesntdivideindex},  we are led to the matrix
$$ \begin{pmatrix}
1 & -1 & 1 &0 & \vdots & 1 \\
0 & 1 & 0 &0 & \vdots & 1 \\
0 & 0 & 0 & 1 & \vdots &  1 \\
0 & 0 & 0 & 0 & \vdots &  -q \\
\end{pmatrix} \ .
$$
Clearly the final row indicates that the system has solutions precisely when ${\rm char}(K)$ divides $q$, i.e., when ${\rm char}(K) \in \mathcal{P}$.    So, by Theorem~\ref{LPA-simple}, the Lie  $K$-algebra $[L_{K}(E_q), L_{K}(E_q)]$ is not simple if and only if  ${\rm char}(K)\in \mathcal{P}$.
\hfill $\Box$
}
\end{example}

We finish this section by presenting, for each prime $p$, an infinite collection of  graphs $E$ for which the  Lie $K$-algebra $[L_K(E),L_K(E)]$ is simple, where $K$ is any field of characteristic $p$. 

\begin{example}\label{ExtendedExamples}
{\rm
For any prime $p$, and any pair of integers $u\geq 2, v\geq 2$, consider the graph $E = E_{u,v,p}$ pictured below.
$$  
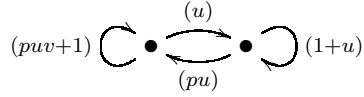
\xymatrix{
  \bullet \ar@(dl,ul)[]^{(puv+1)} \ar@/^/[r]^{(u)}
& \bullet \ar@(ur,dr)[]^{(1+u)} \ar@/^/[l]^{(pu)} }
$$
By Theorem~\ref{SimplicityTheorem},  $L_K(E)$ is a  simple algebra for any field $K$.
Here we have $B_1 = (puv,u)$ and $B_2 = (pu,u)$.  Then $(1,1) \in {\rm span}_K\{B_1,B_2\}$ precisely when we can solve the system 
$$\begin{array}{lcl}   puv k_1+ puk_2 & = &1 \\   \ \ \ uk_1 +   \ \ uk_2& = & 1 \end{array}$$
for $k_1,k_2\in K$.  But clearly the first equation has no solutions in any field of characteristic $p$.   
  Thus, by Theorem~\ref{LPA-simple}, the  Lie algebra $[L_K(E),L_K(E)]$ is simple when ${\rm char}(K)=p,$ as desired.  \hfill  $\Box$
}
\end{example}

In the next section we will show that the Leavitt path algebras associated to the graphs in Example~\ref{ExtendedExamples}  are pairwise non-isomorphic, as well as show that none of these algebras is isomorphic to an algebra of the form $\M_d(L_K(n))$.


\section{Lie algebras arising from purely infinite simple \\ Leavitt path algebras}

  We begin this final section by 
recasting  Theorem~\ref{LPA-comm} in terms of matrix transformations.    For a  finite graph $E$ having $m$ vertices $\{v_1,...,v_m\}$  we let $A_E$ denote the {\it adjacency matrix} of $E$; this is the $m \times m$ matrix whose $(i,j)$ entry is $a_{i,j}$, the number of edges $e$ for which $s(e)=v_i$ and $r(e)=v_j$.    Let $\overline{1}^m$ denote the $m \times 1$ column vector
$(1,1,...,1)^t$ ($t$ denotes `transpose').
Let $B_E$ denote the matrix $A_E^t - I_m$.  In case $E$ has no sinks, $B_E$ is the matrix whose $i$-th column is the element $B_i$ of $K^m$, as in Definition~\ref{number-def}.   Let $B_E^{K^m}$ denote the $K$-linear transformation $K^m \rightarrow K^m$ induced by left multiplication by $B_E$.  (For the remainder of the article we view the elements of $K^m$ as
columns.)
Then, using the notation of Theorem~\ref{LPA-comm}, it is clear that $(1, \dots,1) \in {\rm span}_K\{B_1,...,B_m\}$ if and only if $\overline{1}^m \in {\rm Im}(B_E^{K^m})$.

\begin{definition}\label{MsubEdefinition}
{\rm For a finite  graph $E$ having $m$ vertices we define the  matrix $M_E$ by setting
$$M_E = I_m -A_E^t .$$
(In particular, if $E$ has no sinks, then $M_E = -B_E$.)
For any field $K$ we let  $M_E^{K^m}$ denote the $K$-linear transformation $K^m \rightarrow K^m$ induced by left multiplication by $M_E$.
\hfill $\Box$}
\end{definition}

\begin{remark}\label{imageMversusimageBremark}
Trivially, when $E$ has no sinks, $\overline{1}^m \in {\rm Im}(B_E^{K^m})$ if and only if $\overline{1}^m \in {\rm Im}(M_E^{K^m})$. 

\hfill $\Box$
\end{remark}




\begin{remark}\label{Remarkaboutprimesubfield}
Let $E$ be a finite graph without sinks, and write $E^0 = \{v_1, \dots, v_m\}$. Also, let $K$ be a field with prime subfield $k$.  Then $(1,\dots, 1) \in \mathrm{span}_K\{B_1, \dots, B_m \}$ if and only if $(1,\dots, 1) \in \mathrm{span}_k\{B_1, \dots, B_m \}$ if and only if $(1,\dots, 1)^t$ is in the image of $M_E^{k^m}: k^m \rightarrow k^m$.
This is because solving $M_E \overline{x} = (1,\dots, 1)^t$ for $\overline{x}\in K^m$ amounts to putting into row-echelon form, via row operations,  the matrix resulting from adjoining $(1,\dots, 1)^t$ as a column to $M_E$.  Since the original matrix $M_E$ is integer-valued, all of the entries in the resulting row-echelon form matrix will come from the prime subfield.   Thus in all germane computations we may work over the prime subfield $k$ of $K$.   \hfill $\Box$
\end{remark}

The graphs $E$ for which  $L_K(E)$ is a purely infinite simple algebra have played a central role in the development of the subject of Leavitt path algebras.  A ring $R$ is called {\it purely infinite simple} in case $R$ is a simple ring with the property that each nonzero left ideal $L$ of $R$ contains an {\it infinite idempotent}; that is, a nonzero  idempotent $e$ for which there exist nonzero orthogonal idempotents $f,g$  with the property that  $Re = Rf \oplus Rg$, and $Re \cong Rf$  as left ideals.  (As one consequence,  this means that there are no indecomposable projective left ideals in $R$.)   When $R$ is unital, the purely infinite simplicity property is equivalent to requiring that $R$ not be  a division ring, and that for each $x\neq 0$ in $R$ there exist $\alpha, \beta \in R$ for which $\alpha x \beta = 1$.     (See e.g.\ \cite{AALP}  for the germane definitions, as well as an overview of the main properties of these algebras.)   The key result in this context is

\begin{theorem}[The Purely Infinite Simplicity Theorem] \label{PurelyInfiniteSimple}
Let $K$ be a field, and let  $E$ be a finite graph.  Then  $L_K(E)$ is purely infinite simple if and only if $E$ satisfies the following three properties.

$(1)$  Every vertex $v$ of $E$ connects to  every cycle of $E$.

$(2)$   Every cycle of $E$ has an exit.

$(3)$   $E$ has at least one cycle. 

\noindent

\end{theorem}

If $E$ is a finite graph with the property that every vertex connects to every cycle and every sink of $E$, and $E$ does contain a sink, then necessarily it is the unique sink in $E$, and $E$ must be acyclic.  So, when $E$ is finite,  the statements (1) and (3) of Theorem~\ref{PurelyInfiniteSimple}, taken together, imply condition (1) of Theorem~\ref{SimplicityTheorem}, namely, that every every vertex of $E$ connects to every sink (as there aren't any), and to every infinite path.     So in the context of simple Leavitt path algebras, we get a dichotomy:  either the underlying graph has a (unique) sink (in which case the graph is acyclic), or the graph has at least one cycle.  In case the graph has a sink, the corresponding simple Leavitt path algebra is isomorphic to the complete matrix ring $\mathbb{M}_t(K)$, where $t$ is the number of distinct paths (including the path of length $0$) which end at the sink.  In this situation the corresponding Leavitt path algebra is simple artinian, so that, in particular, every left ideal is artinian, and there exists, up to isomorphism, exactly one indecomposable projective left ideal. On the other hand, the simple Leavitt path algebras arising from graphs containing at least one cycle are purely infinite simple, so that, in particular, no left ideal is artinian, and there exist no indecomposable projective left ideals.   

As the Leavitt algebras $L_K(n)$ (and matrices over them) provide the basic examples of purely infinite simple algebras, it is natural in light of Corollary~\ref{MainCorformatricesoverLeavittalgebras} to investigate the Lie  algebras associated to purely infinite simple Leavitt path algebras.  We do so for the remainder of this article, and in the process 
 provide a broader context for the results of Section~\ref{SimpleLPAandbracketsection}.  
 We start with the following interpretation of Theorem~\ref{LPA-simple}, which follows from Remark~\ref{imageMversusimageBremark}.

\begin{corollary}\label{rewordedCorollary} 
Let $K$ be a field, and let $E$ be a finite graph for which $L_K(E)$ is purely infinite simple.   Then the Lie $K$-algebra $\, [L_K(E),L_K(E)]$ simple if and only if $\, \overline{1}^m \not\in {\rm Im}(M_E^{K^m})$.
\end{corollary}

For any positive integer $j$, we denote the cyclic group of order $j$ by $\mathbb{Z}_j$, while for any prime $p$ we denote the field of $p$ elements by $F_p$.  

We assume from now on that $E$ is a finite graph with $m$ vertices, and we often denote $M_E$ simply by $M$,  for notational convenience. 
The matrix $M$ has historically played  an important role in the structure of purely infinite simple Leavitt path algebras (see \cite[Section 3]{AALP} for more information).   
For instance, since $M$ is integer-valued, we may view left multiplication by $M$ as a linear transformation from $\mathbb{Z}^m$ to $\mathbb{Z}^m$ (we denote this by $M^{\mathbb{Z}^m}$).   Then the Grothendieck group of $L_K(E)$ is given by
$$K_0(L_K(E)) \cong {\rm Coker}(M^{\mathbb{Z}^m}) = \Z^m / {\rm Im}(M^{\Z^m}).$$  (It is of interest to note that the Grothendieck  group of $L_K(E)$ is independent of the field $K$.)  Moreover, under this isomorphism,
$$ [1_{L_K(E)}] \mbox{ in }  K_0(L_K(E)) \ \mbox{  corresponds to }  \ \overline{1}^m+ {\rm Im}(M^{\Z^m}) \mbox{ in } {\rm Coker}(M^{\mathbb{Z}^m}).$$

 For an abelian group $G$ (written additively), an element $g \in G$, and positive integer $j$,   we say $g$ is {\it j-divisible} in case there exists $g' \in G$ for which $g = g' + \cdots + g' $ ($j$ summands).   We use the previous discussion to give  another interpretation of Theorem~\ref{LPA-simple} in the case of purely infinite simple Leavitt path algebras.  We thank Christopher Smith for pointing out this connection. 
 
 \begin{theorem}\label{secondrewordedCorollary}
Let $K$ be a field, let $E$ be a finite graph for which $L_K(E)$ is purely infinite simple, and let $M=M_E$ denote the matrix of Definition~\ref{MsubEdefinition}.

$(1)$   Suppose that $\,{\rm char}(K)=0$.   Then the Lie $K$-algebra $\,[L_K(E),L_K(E)]$ is simple if and only if $\, \overline{1}^m + {\rm Im}(M_E^{\Z^m})$ has infinite order in $\,{\rm Coker}(M_E^{\mathbb{Z}^m});$ that is,  if and only if $\,[1_{L_K(E)}]$ has infinite order in $K_0(L_K(E))$.

$(2)$  Suppose that $\, {\rm char}(K)=p\neq 0$.   Then the Lie $K$-algebra $\, [L_K(E),L_K(E)]$ is simple if and only if $\, \overline{1}^m~+~{\rm Im}(M_E^{\Z^m})$ is not $p$-divisible in $\, {\rm Coker}(M_E^{\mathbb{Z}^m});$ that is,  if and only if $\, [1_{L_K(E)}]$ is not $p$-divisible in $K_0(L_K(E))$.

 \end{theorem}
 \begin{proof}
 (1)  We show that $\overline{1}^m \in {\rm Im}(M_E^{K^m})$ if and only if $\overline{1}^m~+~{\rm Im}(M_E^{\Z^m})$ has finite order in ${\rm Coker}(M_E^{\mathbb{Z}^m})$, from which the statement follows by Corollary~\ref{rewordedCorollary}.  By Remark~\ref{Remarkaboutprimesubfield}, we need only show that $\overline{1}^m \in {\rm Im}(M_E^{\Q^m})$ if and only if $\overline{1}^m + {\rm Im}(M_E^{\Z^m})$ has finite order in ${\rm Coker}(M_E^{\mathbb{Z}^m})$.  If $\overline{1}^m + {\rm Im}(M_E^{\Z^m})$ has finite order in ${\rm Coker}(M_E^{\mathbb{Z}^m})$, then there exists a positive integer $n$ for which $(n,n,\hdots,n)^t \in {\rm Im}(M_E^{\Z^m})$, i.e., there exists $ \overline{z} = (z_1, z_2, \hdots, z_m)^t \in \Z^m$ for which $M_E \overline{z} = (n,n\hdots,n)^t$.   But then $\overline{q} = (\frac{{z_1}}{n} , \frac{ {z_2}}{n} , \hdots, \frac{{z_m}}{n} )^t \in \Q^m$ satisfies $M_E \overline{q} = \overline{1}^m$.     Conversely, if $\overline{1}^m \in {\rm Im}(M_E^{\Q^m})$ then there exists $ (\frac{{z_1}}{n_1} , \frac{ {z_2}}{n_2} , \hdots, \frac{{z_m}}{n_m} )^t \in \Q^m$ with $ \overline{1}^m = M_E (\frac{{z_1}}{n_1} , \frac{ {z_2}}{n_2} , \hdots, \frac{{z_m}}{n_m} )^t $.   If $n = n_1n_2\cdots n_m $, then $(n,n,\cdots,n)^t = M_E (\frac{{z_1n}}{n_1} , \frac{ {z_2n}}{n_2} , \hdots, \frac{{z_mn}}{n_m} )^t  \in {\rm Im}(M_E^{\Z^m})$, so that $(1,1,\cdots,1)^t + {\rm Im}(M_E^{\Z^m})$ has finite order (indeed, order at most $n$) in ${\rm Coker}(M_E^{\mathbb{Z}^m})$.

 (2)  Analogously to the proof of part (1),   we show that $\overline{1}^m \in {\rm Im}(M_E^{K^m})$ if and only if $\overline{1}^m~+~{\rm Im}(M_E^{\Z^m})$ is $p$-divisible in ${\rm Coker}(M_E^{\mathbb{Z}^m})$.  By Remark~\ref{Remarkaboutprimesubfield}, we need only show that $\overline{1}^m \in {\rm Im}(M_E^{F_p^m})$ if and only if $\overline{1}^m + {\rm Im}(M_E^{\Z^m})$ is $p$-divisible in ${\rm Coker}(M_E^{\mathbb{Z}^m})$.   If $\overline{1}^m + {\rm Im}(M_E^{\Z^m})$ is $p$-divisible in ${\rm Coker}(M_E^{\mathbb{Z}^m})$, then there exists $\overline{z} \in \Z^m$ for which $p\overline{z} + {\rm Im}(M_E^{\Z^m}) = \overline{1}^m + {\rm Im}(M_E^{\Z^m})$, i.e., $\overline{1}^m -p\overline{z} \in {\rm Im}(M_E^{\Z^m})$.  Reducing this integer-valued system of equations mod $p$ yields $\overline{1}^m \in {\rm Im}(M_E^{F_p^m})$.  The converse can be proved by reversing this argument.  
  \end{proof}

 \begin{remark} Let $K$ be a field such that ${\rm char}(K)=p \neq 0$.    For Leavitt path algebras of the form $R = {\M}_d(L_K(n))$, we have that $K_0(R) \cong \mathbb{Z}_{n-1}$. Moreover, under this isomorphism  the element $[1_{R}]$ in $K_0(R)$ corresponds to the element $d$ in $\mathbb{Z}_{n-1}$.   Thus the $p$-divisibility of $[1_{R}]$ in $K_0(R)$ is equivalent to  the $p$-divisibility of $d$ in $\mathbb{Z}_{n-1}$, which in turn is equivalent to determining whether or not the linear equation  $px \equiv d$ (mod $n-1$) has solutions.   It is elementary number theory that this equation has solutions precisely when $g.c.d.(p,n-1)$ divides $d$.   So,  by Theorem~\ref{secondrewordedCorollary}(2), we see that $[{\M}_d(L_K(n)), {\M}_d(L_K(n))]$ is simple precisely when $g.c.d.(p,n-1)$ does not divide $d$, which is clearly equivalent to  the statement ``$p$ divides $n-1$ and $p$ does not divide $d$".  This observation provides a broader framework for Corollary~\ref{MainCorformatricesoverLeavittalgebras}. \hfill $\Box$
 \end{remark}

Now continuing our description of various connections between the matrix $M = M_E$ and the Grothendieck group of $L_K(E)$,  we recall from \cite[Section 3]{AALP} that the matrix $M$ can be utilized to determine the specific  structure of $K_0(L_K(E))$ in case $L_K(E)$ is purely infinite simple, as follows.  Given any  integer-valued $d\times d$ matrix $C$, we say that a matrix $C'$ is {\it equivalent} to $C$ in case $C' = PCQ$ for some matrices $P,Q$ which are invertible in $\M_d(\Z)$.  Computationally, this means $C'$ can be produced from $C$ by a sequence of row swaps and column swaps, by multiplying any row or column by $-1$, and by using the operation of adding a $\Z$-multiple of one row (respectively, column) to another row (respectively, column).     The {\it Smith normal form} of an integer-valued $d \times d$ matrix $C$ is the diagonal matrix  which is equivalent to $C$, having diagonal entries $\alpha_1, ..., \alpha_d$, such that, for all nonzero $\alpha_i$ ($1 \leq i<d$), $\alpha_i$ divides $\alpha_{i+1}$.   (The Smith normal form of a matrix always exists. Also, if we agree to write any zero entries last, and to make all $\alpha_i$  nonnegative, then the Smith normal form of a matrix  is unique.)     By the discussion in \cite[Section 3]{AALP},  for a graph $E$ satisfying the properties of Theorem~\ref{PurelyInfiniteSimple}, if $\alpha_1, ..., \alpha_d$ are the diagonal entries of  the Smith normal form of $M_E$, then 
$$K_0(L_K(E))\cong \Z_{\alpha_1}\oplus \cdots \oplus \Z_{\alpha_d}.$$   
With this observation,  we have the tools to justify a statement made in the previous section.

\begin{example} \label{revisittwovertexexample}
{\rm  

Consider again the graphs $E = E_{u,v,p}$ arising in Example~\ref{ExtendedExamples}.  Then 
$$A_{E} = \begin{pmatrix}
puv+1 & u  \\
pu & 1+u  \\
\end{pmatrix},  \ \mbox{so that } \ 
M_{E} = I_2 - A_{E}^t =  \begin{pmatrix}
-puv & -pu  \\
-u & -u  \\
\end{pmatrix}.
$$
\noindent
The Smith normal form of $M_{E}$ is easily computed to be
$$  \begin{pmatrix}
u & 0  \\
0 & pu(v-1)  \\
\end{pmatrix},
$$
implying that $K_0(L_K(E))\cong \mathbb{Z}_u \oplus \mathbb{Z}_{pu(v-1)}$. Thus for any choices of $u,u'$ and $v,v'$ where $u\neq u'$ or $v\neq v'$, we have that $L_K(E_{u,v,p}) \not\cong L_K(E_{u',v',p})$. Furthermore, since $u\geq 2$ and $v\geq 2$, none of these algebras has cyclic $K_0$, so that none of these algebras is isomorphic to an algebra of the form
${\M}_d(L_K(n))$, as claimed. \hfill $\Box$
}
\end{example}

We conclude the article with an observation 
 about the $K$-theory of Leavitt path algebras in the context of their associated Lie algebras.   
An open question in the theory of Leavitt path algebras is the  {\it Algebraic Kirchberg Phillips Question}:  If $E$ and $F$ are finite graphs with the property that $L_K(E)$ and $L_K(F)$ are purely infinite simple, and $K_0(L_K(E)) \cong K_0(L_K(F))$ via an isomorphism which takes $[1_{L_K(E)}]$ to $[1_{L_K(F)}]$, are $L_K(E)$ and $L_K(F)$ necessarily isomorphic?   (See \cite{Flow} for more details.)
Since the property ``the  Lie $K$-algebra  $[R,R]$ is simple" is an isomorphism invariant of a $K$-algebra   $R$, one might look for a possible negative answer to the Algebraic Kirchberg Phillips Question  in this context.  However, by Theorem~\ref{secondrewordedCorollary}, we get immediately the following result.

\begin{proposition}\label{Kzeropairdetermines}
Let $E$ and $F$ be finite graphs, and $K$ any field. Suppose that $L_K(E)$ and $L_K(F)$ are purely infinite simple, and  that $K_0(L_K(E)) \cong K_0(L_K(F))$ via an isomorphism which takes $\, [1_{L_K(E)}]$ to $\, [1_{L_K(F)}]$. Then the Lie $K$-algebra $\, [L_K(E),L_K(E)]$ is simple if and only if the Lie $K$-algebra $\, [L_K(F),L_K(F)]$ is simple.
\end{proposition} 

\subsection*{Acknowledgement}

The authors are grateful to Camilla and David Jordan for helpful comments, and to the referee for a thoughtful and thorough report.

\vspace{.1in}

\noindent
Department of Mathematics \newline
University of Colorado \newline
Colorado Springs, CO 80918 \newline
USA \newline

\noindent {\tt abrams@math.uccs.edu} \newline
\noindent {\tt zmesyan@uccs.edu}

\end{document}